\documentclass[11pt]{amsart}
\usepackage{mathrsfs, amsmath, amssymb,graphics,booktabs}
\usepackage[all]{xy}
\usepackage{verbatim}
\usepackage{color}
\usepackage[title]{appendix}

\newtheorem{theorem}{Theorem}
\newtheorem{lemma}[theorem]{Lemma}
\newtheorem{proposition}[theorem]{Proposition}

\newtheorem{corollary}[theorem]{Corollary}

\theoremstyle{definition}
\newtheorem{definition}[theorem]{Definition}

\theoremstyle{remark}

\numberwithin{equation}{section}

\newcommand{\RNum}[1]{\uppercase\expandafter{\romannumeral #1\relax}}

\begin{document}

\title{Distinguishing 4-dimensional geometries via  profinite completions}

\author{Jiming Ma}
\address{School of Mathematical Sciences, Fudan University, Shanghai, 200433, P. R. China}
\email{majiming@fudan.edu.cn}

\author{Zixi Wang}
\address{School of Mathematical Sciences, Fudan University, Shanghai, 200433, P. R. China}

\email{zxwang18@fudan.edu.cn}



\date{Jan. 4, 2021}

\thanks{Jiming Ma was supported by NSFC No.11771088.}

\begin{abstract}

It is well-known that there are 19 classes of geometries for 4-dimensional manifolds in the sense of Thurston. We could ask that to what extent the geometric information is revealed by the profinite completion of the fundamental group  of a closed  smooth geometric 4-manifold. In this paper, we show that the geometry of a closed orientable 4-manifold in the sense of Thurston could be detected by the profinite completion of its fundamental group except when the geometry is $ \mathbb{H}^{4}$,  $\mathbb{H}^{2}_{\mathbb{C}}$ or $\mathbb{H}^2 \times \mathbb{H}^2$.

Moreover, despite the fact that not every smooth 4-manifold could admit one geometry in the sense of Thurston, some 4-dimensional manifolds with Seifert fibred structures are indeed geometric. For a closed orientable Seifert fibred 4-manifold $M$, we show that whether $M$ is geometric could be detected by the profinite completion of its fundamental group.
\end{abstract}

\maketitle

\section{introduction}
An \textit{n-dimensional geometry} $\mathbb{X}$ in the sense of Thurston means a pair $\left(X,G\right)$, where $X$ is a 1-connected $n$-dimensional smooth Riemann manifold with a Lie group action of $G$ \cite{Fi:1984}. The $G$-action is required to be transitive, effective and isometrical. We could assume $G$ is maximal among all such Lie groups acting on $X$, which could be seen as the isometry group of $X$ for certain Riemann metric. We call a manifold $M$ \textit{admitting geometry $\mathbb{X}$} if it is a quotient $X\slash \Gamma$, where $\Gamma$ is a discrete subgroup of $G$ acting freely on $X$ such that $X\slash \Gamma$ has finite volume. It is well-known that the 3-dimensional geometries are classified by Thurston. There are eight 3-dimensional geometries and six of them can be realized as Seifert fibred spaces. The details are referred to Scott  \cite{Scott:1983}. Filipkiewicz \cite{Fi:1984} classified the 4-dimensional geometries in the sense of Thurston, and there are nineteen classes of maximal geometries with one class including infinitely many non-equivalent geometries. Any closed orientable 4-manifold with eight   among nineteen geometries of Filipkiewicz could be realized as a  4-dimensional Seifert fibred manifold except for two flat 4-manifolds \cite{Ue:1990,Hi:2002}.

It has been long focused that distinguishing finitely generated residually finite group $G$ by the collection of its finite quotients, or equivalently, by its \textit{profinite completions} $\widehat{G}$. Let $\mathscr{C}$ denote a certain collection of residually finite groups. We say a group $G\in\mathscr{C}$ is \textit{profinitely rigid in $\mathscr{C}$} if $G$ could be distinguished  by $\widehat{G}$ in $\mathscr{C}$ up to isomorphism. Otherwise $G$ has \textit{profinite flexibility} when the rigidity fails. If $G$ is distinguished by $\widehat{G}$ in the collection of all finitely generated residually finite groups,  we say that $G$ admits \textit{absolutely profinite rigidity}. Considering about free group $F_n$ of rank $n\ge2$, it's still open that whether $F_n$ is profinitely rigid in the absolute sense \cite{Noskov:1979,Reid:2018}.

In 2018, Reid gave an ICM talk of major problems and recent work about profinite rigidity related to low-dimensional topology \cite{Reid:2018}.  Bridson, Conder and Reid showed that Fuchsian groups are profinitely rigid among all the lattices of connected Lie groups \cite{BCR:2015}.  For a 3-dimensional geometric manifold in the sense of Thurston, its geometry could be distinguished by the profinite completion of its fundamental group proved by Wilton and Zalesskii \cite{Wil:2017}. Bridson, McReynolds, Reid and Spitler showed that there exist arithmetic lattices in $PSL(2,\mathbb{C})$ which are profinitely rigid in the absolute sense \cite{BMRS:2018}. Jaikin-Zapirain proved that being fibred is a profinite invariant for compact orientable aspherical 3-manifolds \cite{J:2020}. Bridson and Reid \cite{BrRe:2015}, Boileau and Friedl \cite{BoFr:2015}  independently proved that the figure-eight knot group could be distinguished by its fundamental group among all the knot groups. However there are examples of 3-manifolds which are not profinitely rigid. Funar \cite{Funar} gave infinite pairs of non-homeomorphic $\mathbb{S}ol^3$ manifolds with isomorphic profinite completions  based on the work of Stebe \cite{Stabe:1972}. Hempel  \cite{Hem:2014} showed that there are non-homeomorphic closed Seifert fibred spaces of geometry $\mathbb{H}^2\times\mathbb{E}$ with isomorphic profinite completions. Wilkes \cite{Wilk:2017} proved that the profinite rigidity holds in the 3-dimensional closed orientable Seifert fibred manifolds except for the case exhibited by Hempel.

Considering the 4-dimensional Thurston's geometries, Stover \cite{Sto:2019} showed the profinite flexibility of $\mathbb{H}_{\mathbb{C}}^2$-manifolds. Piwek, Popovi\'{c} and Wilkes proved that two 4-dimensional $\mathbb{E}^4$-orbifolds are isomorphic if the profinite completions of their fundamental groups are isomorphic, since their fundamental groups could be seen as 4-dimensional crystallographic groups \cite{PPW:2019}.

We focus on the problem that for a closed orientable 4-dimensional manifold, whether its geometry could be distinguished by the profinite completion of its fundamental group. Inspired by Theorem B of \cite{Wil:2017} of Wilton and Zalesskii, we conclude following theorem in this paper.

\begin{theorem}\label{thm:1}
	Let $M$, $N$ be two closed orientable 4-dimensional manifolds with infinite fundamental groups and distinct geometries $\mathbb{X}_1$ and $\mathbb{X}_2$ respectively in the sense of Thurston. Then $\widehat{\pi_1(M)}\ncong\widehat{\pi_1(N)}$ if $\{\mathbb{X}_1,\ \mathbb{X}_2\}\nsubseteq\{\mathbb{H}_{\mathbb{C}}^2,\ \mathbb{H}^4, \ \mathbb{H}^2 \times \mathbb{H}^2\}$.
\end{theorem}

For a closed orientable 4-dimensional Seifert fibred manifold, it may not admit any geometric structure if the monodromy doesn't meet certain conditions. However, whether it is geometric could be detected by the profinite completion of its fundamental group.

\begin{theorem}\label{thm:geo}
	Let $M$, $N$ be two closed orientable 4-dimensional Seifert fibred manifolds over hyperbolic 2-orbifolds, and suppose they have isomorphic profinite completions of their fundamental groups $\widehat{\pi_1(M)}\cong\widehat{\pi_1(N)}$. Then $M$ is geometric if and only if $N$ is geometric.
\end{theorem}

By the proof of Theorem \ref{thm:1} and Theorem \ref{thm:geo}, we have the following corollary to describe the geometry of a 4-dimensional geometric manifold by the profinite completion of its fundamental group.

\begin{corollary}\label{Cor:3}
	Let $M$ be a closed, orientable 4-manifold with geometry $\mathbb{X}$ such that $\mathbb{X}\notin \{\mathbb{H}^4_{\mathbb{R}},\ \mathbb{H}^2_{\mathbb{C}},\ \mathbb{H}^2\times\mathbb{H}^2 \}$, and $\pi_1(M)$ is infinite, then:
	
	(1) $\mathbb{X}$ is $\mathbb{S}^3\times\mathbb{E}$ or $\mathbb{S}^2\times\mathbb{E}^2$ if and only if $\widehat{\pi_1(M)}$ is virtually $\widehat{\mathbb{Z}}$ or  $\widehat{\mathbb{Z}}^2$ respectively;
	
	(2) $\mathbb{X}$ is $\mathbb{E}^4$, $\mathbb{N}il^3\times\mathbb{E}$, $\mathbb{N}il^4$, $\mathbb{S}ol^4_1$ or one of geometries $\{\mathbb{S}ol^4_0,\ \mathbb{S}ol^4_{m,n}\}$ if and only if $\widehat{\pi_1(M)}$ is virtually $\widehat{\mathbb{Z}^4}$, virtually nilpotent class $2$, virtually nilpotent class $3$, virtually solvable length $3$ or virtually solvable length $2$;
	
	(3) $M$ is a Seifert 4-manifold if and only if  $\widehat{\pi_1(M)}$ contains $\widehat{\mathbb{Z}}^2$ as a normal subgroup; 
	
	(4) $\mathbb{X}$ is $\mathbb{H}^3\times\mathbb{E}$ if and only if $\widehat{\pi_1(M)}$ contains $\widehat{\mathbb{Z}}^2$ as a subgroup but does not contain $\widehat{\mathbb{Z}}^3$ as a subgroup;
	
	(5) $\mathbb{X}$ is $\mathbb{H}^2\times\mathbb{S}^2$ if and only if $\widehat{\pi_1(M)}$ is virtually isomorphic to the profinite completion of some surface group.

\end{corollary}

We should remark that in the proof of Theorem \ref{thm:1}, we use the result of Wilton and Zalesskii about profinite completions of $\mathbb{H}^3$-lattice groups \cite{Wil:2017},  which in turn  depends essentially on the dramatic developments of Agol \cite{Agol:2013} and Wise \cite{Wise:2009}. To make this paper to be reasonably self-contained, we introduce three aspects which are profinite groups, 4-dimensional Seifert manifolds and 4-dimensional geometries in the sense of Thurston. The question of distinguishing geometries for 4-manifolds in the sense of Thurston by the profinite completions of fundamental groups has some part left. In the next period, we want to distinguish  lattices with geometries $\mathbb{H}^4$, $\mathbb{H}_{\mathbb{C}}^2$ and $\mathbb{H}^2\times\mathbb{H}^2$ by the profinite completions of their fundamental groups.

\section{4-dimensional Thurston's geometries}

For a geometry $(X,G)$ and any point $x$ of the $n$-manifold $X$, the stabilizer $G_x$ of the transitive, effective and isometric $G$-action on $X$ is compact. In fact, $G_x$ is isomorphic to a closed subgroup of $O(n)$ since the action is isometric. By classifying all the possible closed subgroups of $SO(4)$, Filipkiewicz \cite{Fi:1984} proved that there are 19 classes of maximal geometries in 4-dimension listed in Table 1.


\begin{table}[htbp]\label{Table}
	\caption{Nineteen  classes of 4-dimensional  geometries  of Filipkiewicz}
	\renewcommand\arraystretch{1.2}
	\begin{tabular}{c|c}
		\hline
		type &  geometries  \\
		\hline

		solvable type  & $\mathbb{E}^{4}$, $\mathbb{N}il^{3}\times \mathbb{E}$, $\mathbb{N}il^{4}$,  $\mathbb{S}ol_{m,n}^{4}$, $\mathbb{S}ol^{4}_{0}$, $\mathbb{S}ol^{4}_{1}$\\
		\hline

		product type &  $ \mathbb{H}^2\times\mathbb{S}^2$,  $\mathbb{S}^3\times\mathbb{E}$, $\mathbb{S}^2\times\mathbb{E}^2$, $\mathbb{H}^{3}\times \mathbb{E}$, $\widetilde{\mathbb{SL}}\times \mathbb{E}$, $ \mathbb{H}^{2}\times \mathbb{E}^{2}$, $\mathbb{H}^{2}\times \mathbb{H}^{2}$ \\	
		
		\hline
	
		hyperbolic &  $ \mathbb{H}^{4}$, $\mathbb{H}^{2}_{\mathbb{C}}$ \\
		\hline	
		finite group  & $\mathbb{S}^4$, $\mathbb{C}P^2$, $\mathbb{S}^2\times \mathbb{S}^2$  \\
		\hline
		$\mathscr{T}\mathbb{H}^{2}$   &  $\mathbb{F}^4$ \\

		\hline
	\end{tabular}

\end{table}

The first type of geometries  are solvable Lie type geometries including  $\mathbb{E}^{4}$, $\mathbb{N}il^{3}\times \mathbb{E}$, $\mathbb{N}il^{4}$,  $\mathbb{S}ol_{m,n}^{4}$, $ \mathbb{S}ol^{4}_{0}$ and $\mathbb{S}ol^{4}_{1}$. There are countable infinitely many maximal geometries belonging to the $ \mathbb{S}ol_{m,n}^{4}$-class. Suppose $M$ is a closed 4-manifold with $\chi(M)=0$, then Theorem $8.1$ of Hillman \cite{Hi:2002} gives the equivalent conditions for $M$ to be an infravable manifold. Here \textit{infrasolvable manifold} means that  $M$ can be seen as a quotient $X/\Gamma$, where $X$ is a 1-connected solvable Lie group and $\Gamma$ is a closed torsion-free subgroup of $Aff(X)=X\rtimes Aut(X)$.

The 3-dimensional geometry $\mathbb{N}il^3$ could be seen as the group of upper triangular matrices  $\{[r,s,t]=\begin{pmatrix}
1&r&t\\0&1&s\\0&0&1
\end{pmatrix}|\ r,s,t\in\mathbb{R}\}$ which has abelianization $\mathbb{R}^2$ represented by $r,s$ and centre $\mathbb{R}$ represented by $t$. So a $\mathbb{N}il^3$-manifold naturally admits a 3-dimensional Seifert fibred structure. The elements $[1,0,0]$, $[0,1,0]$ and $[0,0,1\slash q]$ generate a discrete cocompact subgroup of $\mathbb{N}il^3$ which is isomorphic to $\Gamma_q$ defined in (\ref{3}) later for any integer $q>1$. For a closed 4-manifold $M$ admitting geometry $\mathbb{N}il^3\times\mathbb{E}$, there is a finite cover $M'$ of $M$ such that $\pi_1(M')\cong \Gamma_q\times \mathbb{Z}$ for some integer $q>1$, which means that $\pi_1(M)$ has virtually nilpotent class 2. The definition of nilpotent class is in Section 6.

The geometry $\mathbb{N}il^{4}$ is the semi-direct product $\mathbb{R}^{3}\rtimes_{\theta}\mathbb{R}$ where $\theta=[t,t,\frac{t^2}{2}]$ and the component of its isometry group with identity is $\mathbb{N}il^{4}$ itself as left translation. It has abelianization $\mathbb{R}^{2}$ and  typical central series consisting of centre $\zeta\mathbb{N}il^{4}$ and commutator subgroup $(\mathbb{N}il^4)'$ such that  $\zeta\mathbb{N}il^{4}\cong\mathbb{R}<(\mathbb{N}il^4)'\cong\mathbb{R}^2$. For a closed orientable $\mathbb{N}il^4$-manifold, its fundamental group has nilpotent class 3.

The 3-dimensional geometry $\mathbb{S}ol^3$ is the Lie group defined as the semi-direct product $\mathbb{R}^2\rtimes_{\varphi}\mathbb{R}$ such that $\varphi(t)=diag[e^t,e^{-t}]$. A closed orientable $\mathbb{S}ol^3$-manifold inherits a natural 2-dimensional foliation such that the manifold is actually a $T^2$-bundle over $S^1$. If a closed 4-manifold $M$ admits geometry $\mathbb{S}ol^3\times\mathbb{E}$, it naturally admits a 4-dimensional Seifert fibred structure over a euclidean 2-orbifold (see Section 3).

The geometry $\mathbb{S}ol_{m,n}^{4}$ represents  the semi-direct product $\mathbb{R}^{3}\rtimes_{\theta_{m,n}}\mathbb{R}$ where $\theta_{m,n}(t)=[e^{at},e^{bt},e^{ct}]$ such that $e^a$, $e^b$ and $e^c$ are three distinct roots (with $a<b<c$ real) of $f(x)=x^3-mx^2+nx-1$ where $m,n$ are positive integers. We could always assume $m\leq n$ since $\mathbb{S}ol_{m,n}^{4}$ is isomorphic to $\mathbb{S}ol_{n,m}^{4}$. In fact, $m$ and $n$ satisfy the condition $0<2\sqrt{n}\leq m\leq n$ since $\theta_{m,n}$ has three distinct real roots. Under this notation, the geometry $\mathbb{S}ol_{m,n}^{4}$ is equivalent to $\mathbb{S}ol_{m',n'}^{4}$ if and only if there exists $\lambda\in\mathbb{R}$ such that $(a,b,c)=\lambda(a',b',c')$.
The fundamental group of a closed orientable $\mathbb{S}ol_{m,n}^4$-manifold has the form of $\mathbb{Z}^3\rtimes_B\mathbb{Z}$, where $B\in GL(3,\mathbb{Z})$ with three distinct real eigenvalues.
When $m=n$, $\theta_{m,m}$ has one and only one eigenvalue of $1$. Hence $\mathbb{S}ol_{m,m}^{4}$ is actually $\mathbb{S}ol^{3}\times\mathbb{E}$ where the manifolds admit Seifert fibred structures.

The geometry $\mathbb{S}ol_0^4$ is the semi-direct product $\mathbb{R}^{3}\rtimes_{\xi}\mathbb{R}$ where $\xi(t)$ denotes the diagonal matrix $diag[e^t,e^t,e^{-2t}]$. For a closed orientable $\mathbb{S}ol_0^4$-manifold, its fundamental group is the semi-direct product $\mathbb{Z}^3\rtimes_B\mathbb{Z}$, where $B\in GL(3,\mathbb{Z})$ with one real eigenvalue and two conjugate complex eigenvalues.

The general method to describe the fundamental group of 4-manifold in this paper follows Hillman's book \cite{Hi:2002}. We use Hirsch-length to measure solvable groups. A \textit{virtually polycyclic} group $G$ is one which has a finite index subgroup with a composition series such that each quotient is infinite cyclic. Then the number of infinite cyclic factors is independent of the choices of finite index subgroup and the composition series, which is defined as \textit{Hirsch-length} $h(G)$ of $G$. For example,
\begin{align}\label{3}
\Gamma_q=\left\langle x,y,z|\ xz=zx,\ yz=zy,\ xy=z^q yx \right\rangle
\end{align}  for any natural number $q>1$ is a typical virtually polycyclic group such that $h(\Gamma_q)=3$, and it's important in the $\mathbb{N}il^3\times\mathbb{E}$ and $\mathbb{S}ol^4_1$ geometries. A closed 4-manifold $M$ with one of the six solvable Lie geometries always satisfies $h(\pi_1(M))=4$ by Theorem 8.1 of \cite{Hi:2002}.

The $\mathbb{S}ol_1^4$-geometry is the group of real matrices $\{\begin{pmatrix}
1&x&z\\0&t&y\\0&0&1
\end{pmatrix}|\ x,y,z,t\in\mathbb{R},\ t>0 \}$. The fundamental group of a closed orientable $\mathbb{S}ol_1^4$-manifold $M$ is an extension of $\Gamma_q$ by $\mathbb{Z}$ for some integer $q>1$. The extension is in fact a semi-direct product $\Gamma_q\rtimes\mathbb{Z}$, and $M$ is the corresponding mapping torus of some $\mathbb{N}il^3$-manifold with fundamental group $\Gamma_q$. A closed orientable 4-manifold with geometry $\mathbb{N}il^3\times\mathbb{E}$ or $\mathbb{N}il^4$ could also be realized as a mapping torus of a  $\mathbb{N}il^3$-manifold, and its geometry depends on the automorphism of $\Gamma_q$.

The following is concluded from Chapter 8 of Hillman's \cite{Hi:2002}. Let $G$ be a group, we use $\sqrt{G}$ to denote its \textit{nilradical}, which means its maximal locally nilpotent normal subgroup of $G$. Here \textit{locally nilpotent} means that every finitely generated subgroup of $\sqrt{G}$ is nilpotent.  For a closed 4-dimensional geometric manifold $M$, its geometry depends largely on $\sqrt{\pi_1(M)}$, which is the maximal nilpotent normal subgroup of $\pi_1(M)$. A 4-manifold $M$ admits geometry $\mathbb{E}^4$ if and only if $\sqrt{\pi_1(M)}\cong\mathbb{Z}^4$, while it has geometry $\mathbb{N}il^3\times\mathbb{E}$ or $\mathbb{N}il^4$ if $\sqrt{\pi_1(M)}$ has nilpotent class of 2 or 3 respectively. When $M$ admits geometry $ \mathbb{S}ol_{m,n}^{4}$ or $ \mathbb{S}ol^{4}_{0}$, the nilradical $\sqrt{\pi_1(M)}\cong\mathbb{Z}^3$, and $M$ could also be seen as the mapping torus of $\mathbb{R}^3\slash\mathbb{Z}^3$ when $m\neq n$. The manifold with geometry $\mathbb{S}ol^4_1$ has $\Gamma_q$ as its nilradical for some integer $q>1$.

When the manifold admits one of geometries $\mathbb{S}^4$, $\mathbb{C}P^2$ and $\mathbb{S}^2\times \mathbb{S}^2$, our problem on profinite completion detecting geometry is out of consideration since these geometries have compact models.  
Finally the  geometry $\mathbb{F}^4$ is the tangent space $\mathscr{T}\mathbb{H}^2$ of the hyperbolic plane $\mathbb{H}^2$ which can not be realized by any closed 4-manifold, so this geometry is also excluded.

\section{Seifert fibred 4-manifolds}

By a closed orientable \textit{Seifert fibred 4-manifold} $M$, it means that it is the total space of a bundle  $\pi:M\rightarrow B$ over a closed $2$-orbifold $B$ such that the general fibre is the $2$-torus $T^2$. To describe all the Seifert 4-manifolds, there are some invariants which are similar to the case in 3-dimensional Seifert fibred manifolds \cite{Ue:1990,Hi:2002}.

Locally, a point $p$ on the base 2-orbifold $B$ has a neighbourhood $D=D^2/G$  where $D^2$ is a 2-disc centred at $0\in\mathbb{R}^2$ and $G$ is a discrete subgroup of $O(2)$ corresponding to the stabilizer of $p$. 
Then $\pi^{-1}(D)$ is identified with $T^2\times D^2\slash G$ where the action of $G$ on $T^2\times D^2$ is free. The $G$-action on $T^2\times D^2$ is some lift of $G$-action on $D^2$ so that $\pi|_{\pi^{-1}(D)}$ is the map $T^2\times D^2\slash G\to D^2\slash G$ induced from the natural projection $T^2\times D^2\to D^2$. Here $T^2$ is identified with $\mathbb{R}^2\slash\mathbb{Z}^2$ and the point of $T^2\times D^2$ is represented as $(x,y,z)$ with $(x,y)\in\mathbb{R}^2\slash \mathbb{Z}^2$ and $z\in\mathbb{C}$, $|z|\leq 1$.
 
If $G$ is trivial, we say the point $p\in B$  is a \textit{non-singular point}, and the preimage of $D$ in the total space is just $T^2\times D^2$.

If $G=\mathbb{Z}_{m}$ for $m\geq2$ where the generator $\rho$ acts on $T^2\times D^2$ by $\rho(x,y,z)=(x-\frac{a}{m},\ y-\frac{b}{m},\ e^{\frac{2\pi i}{m}}z)$ with $g.c.d.(m,a,b)=1$. Then $p$ is a \textit{cone point} of $B$ and the fibre over $p$ is called a \textit{multiple torus of type $(m,a,b)$}.

If $G=\mathbb{Z}_{2}$ where the generator $\iota$ acts on $T^2\times D^2$ by $\iota(x,y,z)=(x+\frac{1}{2},\ -y,\ \overline{z})$, then $p$ is on the reflector circle of $B$ and $\pi^{-1}(D)$ is a twisted $D^2$-bundle over the Klein bottle $K$. In this case, the underlying space $|B|$ is a surface with boundary, and each connected component of boundary has a reflector structure of $B$. So $B$ is still a closed orbifold.

If $G$ is the dihedral group $D_{2m}=\left\langle \iota,\ \rho |\ \iota^2=\rho^m=1, \iota\rho\iota^{-1}=\rho^{-1}\right\rangle$ where the generators $\iota$ and $\rho$ act on $T^2\times D^2$ by $\iota(x,y,z)=(x+\frac{1}{2},\ -y,\ \overline{z})$ and $\rho(x,y,z)=(x,\ y-\frac{b}{m},\ e^{\frac{2\pi i}{m}}z)$ where $g.c.d.(m,b)=1$. In this case, $p$ is a \textit{corner reflector} of angle $\pi\slash m$ and the fibre over $p$ is called a \textit{multiple Klein bottle of type $(m,0,b)$}. It means that the base orbifold has reflector circles such that each circle contains finitely many corner reflectors.  From now on, we don't deal with the Seifert 4-manifolds with reflector circles on purpose, for each of them has a canonical double cover which is a Seifert 4-manifold without reflector circles.

Globally, there still needs some invariants to describe the fibration. Suppose the base orbifold $B$ has no reflector circles, then we use the 4-dimensional Seifert invariants defined by Ue \cite{Ue:1990}:

(1) the \textit{monodromy matrices} $A_i,B_i\in SL(2,\mathbb{Z})$ satisfying $\prod_{i=1}^{g}[A_i,B_i]=I$ along the collection of standard generators $s_i,t_i\ (i=1,2,...,g)$ of $\pi_1(|B|)$, where the underlying space $|B|$ of $B$ is an orientable surface of genus $g$.

(1') the \textit{monodromy matrices} $A'_i\in GL_2(\mathbb{Z})$ along the collection of standard generators $v_i\ (i=1,2,...,g')$ of $\pi_1(|B'|)$, where the underlying space $|B'|$ of $B'$ is a non-orientable surface of genus $g'$.

(2) the tuple $(m_i,a_i,b_i)$ to describe the fibre type over singular point $p_i\ (i=1,...,t)$ of the base orbifold.

(3) the \textit{obstruction} $(a',b')\in\mathbb{Z}^2$. Let $q_i$ be the lift of the meridian circle centred at a singular point $p_i$, then $(a',b')$ is the obstruction to extend $\cup q_i $ to the cross section in $\pi^{-1}(B-\cup$(the disk neighbourhood of $p_i))$.

(4) if all the monodromies are trivial, the euler number could be defined as $e=(a'+\Sigma a_i/m_i,\ b'+\Sigma b_i/m_i)\in\mathbb{Q}^2$ (mod the action of $GL(2,\mathbb{Z})$) which is similar to the definition of rational euler number in 3-dimensional Seifert fibred space. The details are in \cite{Ue:1990}.

For example of the simplest case, we can describe $M$, a $T^2$-bundle over $T^2$, as $\left\lbrace A,B,(m,n)\right\rbrace$ for $ABA^{-1}B^{-1}=I$ and $m,n\in\mathbb{Z}$. Then the fundamental group of $M$ is
\begin{center}
	$\pi_{1}(M)=\left\langle l,h,s,t|\ [l,h]=1,\ s(l,h)s^{-1}=(l,h)A,\ t(l,h)t^{-1}=(l,h)B,\ [s,t]=l^m h^n\right\rangle $,
\end{center}
where $l,h$ are the generators of the general fibre $T^2$ and $s,t$ are the lift of generators of base $T^2$. The representation is not unique since there is a certain transformation between two sets of Seifert invariants of one Seifert 4-manifold.

By classifying closed orientable 4-dimensional Seifert fibred manifolds using the 4-dimensional Seifert invariants, Ue connected these manifolds to Thurston's geometries:

\begin{theorem}\cite{Ue:1990}
	Let $M$ be a closed orientable 4-manifold which is Seifert fibred over a 2-orbifold B, then
	
	(1) $B$ is spherical or bad if and only if $M$ admits geometry $\mathbb{S}^3\times\mathbb{E}$ or $\mathbb{S}^2\times\mathbb{E}^2$;
	
	(2) $B$ is flat if and only if $M$ admits one of geometries $\mathbb{E}^{4}$, $\mathbb{N}il^{4}$, $\mathbb{N}il^{3}\times \mathbb{E}$ or $\mathbb{S}ol^3\times\mathbb{E}$ except for two flat manifolds which are not Seifert \cite{Hi:2002};
	
	(3)  $B$ is hyperbolic if and only if $M$ admits one of geometries $\widetilde{\mathbb{S}L_2}\times \mathbb{E}$ and $ \mathbb{H}^{2}\times \mathbb{E}^{2}$, or $M$ is non-geometric.
\end{theorem}

\begin{theorem}\cite{Ue:1990}
	Let $M$, $N$ be two closed orientable Seifert 4-manifolds over aspherical bases with isomorphic fundamental groups, then $M$ is diffeomorphic to $N$. Moreover, if the base orbifolds of $M$ and $N$ are hyperbolic or $M$ and $N$ admit geometry  $\mathbb{N}il^{4}$ or $\mathbb{S}ol^3\times\mathbb{E}$, then the diffeomorphism is fibre-preserved between $M$ and $N$.
\end{theorem}

However, it needs to be clarified that the Seifert fibration structures of a Seifert 4-manifold may not be unique when the base orbifold is flat, spherical or bad, and the examples with geometries $\mathbb{E}^{4}$ and $\mathbb{N}il^{3}\times \mathbb{E}$ are exhibited in \cite{Ue:1988}. There also exist 4-dimensional Seifert manifolds over hyperbolic base such that they can not admit any 4-dimensional geometries. From the view of representing the Seifert fibred 4-manifolds with Seifert invariants, there is Theorem B of \cite{Ue:1991} by Ue:

\begin{theorem}
	A Seifert 4-manifold $M$ over an orientable hyperbolic base orbifold B admits a geometric structure of type $\mathbb{X}$ if and only if $M$ satisfies one of the following conditions:
	
	(1) All the monodromies are represented by powers of a common periodic matrix in $SL(2,\mathbb{Z})$ or all the monodromies are trivial and the rational euler number is zero. In this case $\mathbb{X}=\mathbb{H}^2 \times \mathbb{E}^2$;
	
	(2) All the monodromies are trivial and the rational euler number is non-zero. In this case $\mathbb{X}=\widetilde{\mathbb{S}L_2}\times \mathbb{E}$.
\end{theorem}

On the other hand, the condition of $M$ being geometric could also be rephrased from the view of groups by Hillman, and we use the definition as follows \cite{Hi:2002}.

\begin{theorem} \cite[Theorem 9.2]{Hi:2002}
	Let $M$ be a smooth closed 4-manifold with fundamental group $\pi$. Then $M$ is aspherical with $\sqrt{\pi}\cong \mathbb{Z}^2$ if $h(\sqrt{\pi})=2$, $[\pi:\sqrt{\pi}]=\infty$ and $\chi(M)=0$.
\end{theorem}

\begin{corollary}\label{cor:hyperbolic}\cite[Corollary 9.2.1]{Hi:2002}
	Let $M$ be a smooth closed 4-manifold with fundamental group $\pi$, then $M$ is homotopy equivalent to a Seifert fibred space with general fibre $T^2$ or Klein bottle over a hyperbolic 2-orbifold if and only if $h(\sqrt{\pi})=2$, $[\pi:\sqrt{\pi}]=\infty$ and $\chi(M)=0$.
\end{corollary}

Let $C_{\pi}(\sqrt{\pi})$ denote the centre of $\sqrt{\pi}$ in $\pi$, and there are following theorems.

\begin{theorem}\label{thm:h2*e2}\cite[Theorem 9.5]{Hi:2002}
	Let M be a smooth closed 4-manifold with fundamental group $\pi$, then the following conditions are equivalent:
	
	(1) $M$ is homotopy equivalent to a $\mathbb{H}^2 \times \mathbb{E}^2$-manifold;
	
	(2) $\sqrt{\pi}\cong \mathbb{Z}^2$, $[\pi:\sqrt{\pi}]=\infty$, $[\pi:C_{\pi}(\sqrt{\pi})]<\infty$ , $e^{\mathbb{Q}}(\pi)=0$ and $\chi(M)=0$.
\end{theorem}
\begin{theorem}\label{thm:sl*e} \cite[Theorem 9.6]{Hi:2002}
	Let $M$ be a smooth closed  4-manifold with fundamental group $\pi$, then the following conditions are equivalent:
	
	(1) $M$ is homotopy equivalent to a $\widetilde{\mathbb{S}L_2}\times \mathbb{E}$-manifold;
	
	(2) $\sqrt{\pi}\cong \mathbb{Z}^2$, $[\pi:\sqrt{\pi}]=\infty$, $[\pi:C_{\pi}(\sqrt{\pi})]<\infty$, $e^{\mathbb{Q}}(\pi)\neq0$ and $\chi(M)=0$.
\end{theorem}

Here the \textit{rational euler class} $e^{\mathbb{Q}}(\pi)$ is defined following the definition in Chapter 7 of Hillman's \cite{Hi:2002}. For a 4-dimensional Seifert fibred manifold $M$ over a good or hyperbolic 2-orbifold $B$, its fundamental group  could be represented as a short exact sequence:
\begin{center}
	$1\longrightarrow \mathbb{Z}^2 \longrightarrow \pi_1(M)\longrightarrow\pi_1^{orb}(B)\longrightarrow1$,
\end{center}
where $\pi_1^{orb}(B)$ denotes the orbifold fundamental group of $B$. Consider $\pi=\pi_1(M)$ and $\alpha:\pi_1^{orb}(B)\to Aut(\mathbb{Z}^2)\cong GL_2(\mathbb{Z})$ be the homomorphism induced by conjugation, and $A=\sqrt{\pi}\cong\mathbb{Z}^2$ representing the fibre. Consider the corresponding $\mathbb{Q}[\pi\slash A]$-module $\mathcal{A}=\mathbb{Q}\otimes_{\mathbb{Z}}A$, and $e^{\mathbb{Q}}(\pi)\in H^2(\pi\slash A;\mathcal{A})$
represents the cohomology class corresponding to $\pi$ as an extension of $\pi\slash A$ by $A$. It should be remarked that whether $e^{\mathbb{Q}}(\pi)$ is zero is invariant when pass to a finite index subgroup $\pi'$ containing $A$ such that $\pi'\slash A$ is a surface group. Comparing with the rational euler number defined by Ue in \cite{Ue:1990,Ue:1991}, the rational euler class $e^{\mathbb{Q}}(\pi)$ is well-defined even when all the monodromies are trivial, which provides another tool.

\section{Introduction of profinite completions }

Now we introduce some elementary results of profinite groups based on the book of Ribes and Zalessikii \cite{RiZa:2004}.
\begin{definition}
	For a partially ordered set $I$, an \textit{inverse system} is a family of sets $\{X_i\}_{i\in I}$, and a family of maps $\phi_{ij}:X_i\rightarrow X_j$ whenever $i\leq j$ such that:
	
	(1) $\phi_{ii}=id_{X_i}$,
	
	(2) $\phi_{jk}\phi_{ij}=\phi_{ik}$ if $i\leq j\leq k$.
	
Denoting the inverse system as $(X_i,\phi_{ij},I)$, its \textit{inverse limit} is defined as
\begin{center}
	$\varprojlim X_i=\{(x_i)\in\prod_{i\in I} X_i|\ \phi_{ij}(x_i)=x_j \ if \ i\leq j\}$.
\end{center}

\end{definition}

\begin{definition}
	Given an inverse system $(X_i,\phi_{ij},I)$ such that $X_i$ are finite groups and $\phi_{ij}$ are group homomorphisms, the resulting inverse limit $\varprojlim X_i$ is called a \textit{profinite group}. For a discrete group $G$, the collection $\mathcal{N}$ of its finite index normal subgroups could form an inverse system by declaring that $N_i\leq N_j$ whenever  $N_i,N_j\in\mathcal{N}$ and $N_i\subseteq N_j$. So the group homomorphisms are natural epimorphisms $\phi_{ij}:G/N_i\rightarrow G/N_j$. We call the inverse limit of the inverse system $(G/N_i,\phi_{ij},\mathcal{N})$  the \textit{profinite completion} of $G$ and denote it as $\widehat{G}$.

\end{definition}

There is another important property to introduce:

\begin{definition}
	A group $G$ is \textit{residually finite} if for every non-trivial $g\in G$, there exists some finite index normal subgroup $N\triangleleft G$ such that $g\notin N$.
\end{definition}

A group $G$ is residually finite means that $\bigcap_{N_i\in \mathcal{N}} N_i=1$, and $G$ is residually finite if and only if $\iota:G\rightarrow \widehat{G}$ is injective where $\iota$ is induced by the natural map $G\rightarrow G\slash N_i$ for every $N_i\in\mathcal{N}$.

Considering the discrete topology of the quotients $G/N_i$, the topology on $\widehat{G}$ could be induced as the subtopology of the product topology on $\prod_{N_i\in\mathcal{N}} G/N_i$. Then the profinite completion $\widehat{G}$ is a compact, Hausdorff and totally disconnected group. An open subgroup $H$ of $\widehat{G}$ in the profinite topology means that $H$ is closed of finite index in the sense of topological groups, and $H$ is closed if and only if it is the intersection of all open subgroups of $\widehat{G}$ containing $H$.  On the other hand, when considering the profinite topology on $\widehat{G}$,  the connection between the discrete group and its profinite completion has been made:

\begin{proposition}\label{profinite1-1}\cite{RiZa:2004}
	If $G$ is a finitely generated residually finite group, then there is a one-to-one correspondence between the set $\mathcal{X}$ of subgroups of $G$ with finite index and the set $\mathcal{Y}$ of all open subgroups of $\widehat{G}$. Identifying $G$ with its image in $\widehat{G}$, the correspondence is given by:
	
	(1) for $H\in \mathcal{X}$, $H\longmapsto \overline{H}$,
	
	(2) for $Y\in \mathcal{Y}$, $Y\longmapsto Y\cap G$.
	
	Here $\overline{H}$ means the closure of $H$ under the profinite topology of $\widehat{G}$. Further more, if $K,H\in \mathcal{X}$ and $K\le H$ then $[H:K]=[\overline{H}:\overline{K}]$; $H$ is normal in $G$ if and only if $\overline{H}$ is normal in $\widehat{G}$; and finally, for $H\in \mathcal{X}$, $\overline{H}\cong\widehat{H}$.
\end{proposition}

It is clear that $G$ and $\widehat{G}$ have the same set of finite quotients. Let $\mathcal{C} (G)$ denote the collection of finite quotients of $G$, then we have the following theorem proved in \cite{RiZa:2004}:

\begin{theorem}
	Let $G_1$ and $G_2$ be two finitely generated abstract groups, then $\widehat{G_1}\cong \widehat{G_2}$ if and only if $\mathcal{C} (G_1)=\mathcal{C} (G_2)$.
\end{theorem}

The homology and cohomology of one's profinite completion maintain many features from the discrete group. Serre defined goodness in \cite{Se:2013}:

\begin{definition}
	A finitely generated group $G$ is \textit{good} if for any finite $G$-module $A$, the natural homomorphism of cohomology groups $H^{n}(\widehat{G};A) \longrightarrow H^{n}(G;A) $ induced by $\iota :G\longrightarrow \widehat{G}$ is an isomorphism for any dimension $n$.
\end{definition}

There are some examples for illustration. It is shown that the finitely generated Fuchsian groups are good \cite{GJZZ:2008}, and the extension of a good group by another good group is itself good. It is one application of goodness as follows \cite{GJZZ:2008}:

\begin{lemma}\label{extension}
	Let G be a residually finite, good group and $\phi:H\rightarrow G$ be a surjective homomorphism with residually finite and finitely generated kernel $K$. Then $H$ is residually finite.
\end{lemma}

Segal proved in \cite[Theorem 1]{Segal} that
\begin{theorem}\label{polycyclic}
	Polycyclic-by-finite groups are residually finite.
\end{theorem}

There is also the famous theorem of residual finiteness by Malcev \cite{Malcev:1940}:
\begin{theorem}\label{linear}
	A finitely generated linear group is residually finite.
\end{theorem}
So we could conclude the residual finiteness  we need in this paper:
\begin{corollary}
	Let $M$ be a closed orientable 4-dimensional manifold with one of Thurston's geometries, then $\pi_1(M)$ is residually finite.
\end{corollary}
\begin{proof}
	For a closed 4-manifold $M$ with $\mathbb{H}^2\times\mathbb{H}^2$ geometry, we call $M$  \textit{irreducible} if it can not be covered by the product of two closed orientable surface by the definition in Chapter 14 of Hillman's \cite{Hi:2002}, otherwise $M$ is reducible. If $M$ admits one of geometries $\mathbb{H}_{\mathbb{C}}^2$ and $\mathbb{H}^4$, or $M$ is an irreducible $\mathbb{H}^2\times\mathbb{H}^2$-manifold, then $\pi_1(M)$ is a finitely generated linear group and hence  residually finite by Theorem \ref{linear}. If $M$ admits geometry of solvable Lie type including $\mathbb{E}^{4}$, $\mathbb{N}il^{4}$, $\mathbb{N}il^{3}\times \mathbb{E}$, $ \mathbb{S}ol_{m,n}^{4}$, $\mathbb{S}ol^{4}_{0}$ and $\mathbb{S}ol^{4}_{1}$, then $\pi_1(M)$ is virtually polycyclic, which means that $\pi_1(M)$ is residually finite by Theorem \ref{polycyclic}. Finally if $M$ admits one of geometries $\mathbb{H}^{3}\times \mathbb{E}$, $ \widetilde{\mathbb{S}L_2}\times \mathbb{E}$, $ \mathbb{H}^{2}\times \mathbb{E}^{2}$ and $\mathbb{H}^2\times\mathbb{S}^2$, or $M$ is a reducible $\mathbb{H}^2\times\mathbb{H}^2$-manifold, then $\pi_1(M)$ is virtually an extension of a  finitely generated residually finite group by another residually finite group, leading to the residual finiteness of $\pi_1(M)$ itself by Lemma \ref{extension}. The geometries of compact type are out of consideration for they have compact models. Hence the corresponding fundamental groups are finite, and naturally residually finite.	
\end{proof}

\section{Detecting Seifert 4-manifolds to be geometric by profinite completions}

Firstly we concentrate on the proof of Theorem \ref{thm:geo} to detect whether a 4-dimensional Seifert manifold is geometric by the profinite completion of its fundamental group.

\begin{proof}[\textbf{Proof of Theorem \ref{thm:geo}}]
	
As assumption, both $M$ and $N$ are closed, orientable 4-dimensional Seifert fibred manifolds over hyperbolic 2-orbifolds $B_1$ and $B_2$ with isomorphism  $\widehat{\phi}:\widehat{{\pi_1(M)}}\to\widehat{{\pi_1(N)}}$. The following is to prove that $M$ is geometric if and only if $N$ is geometric.

Comparing Corollary \ref{cor:hyperbolic}, Theorem \ref{thm:h2*e2} and Theorem \ref{thm:sl*e}, we could easily conclude that $M$ is geometric if and only if $[\pi_1(M):C_{\pi_1(M)}(\sqrt{\pi_1(M)})]<\infty$. This condition means that the action of $\alpha:\pi_1^{orb}(B_1)\to Aut(\mathbb{Z}^2)\cong GL_2(\mathbb{Z})$ has finite image in $GL_2(\mathbb{Z})$ \cite[Page 146]{Hi:2002}. Both nilradicals $\sqrt{\pi_1(M)},\ \sqrt{\pi_1(N)}$ are  isomorphic to $\mathbb{Z}^2$, representing the fibres of $M$ and $N$ respectively, and $[\pi_1(M):\sqrt{\pi_1(M)}]=[\pi_1(N):\sqrt{\pi_1(N)}]=\infty$. The core of this proof is to prove that $[\pi_1(M):C_{\pi_1(M)}(\sqrt{\pi_1(M)})]<\infty$ if and only if  $[\widehat{\pi_1(M)}:C_{\widehat{\pi_1(M)}}(\overline{\sqrt{\pi_1(M)}})]<\infty$. We remark that $\overline{\sqrt{\pi_1(M)}}=\widehat{{\sqrt{\pi_1(M)}}}$ by the goodness of $\pi_1(M)$.

 \noindent\textbf{Claim 1:}
	The isomorphism $\widehat{\phi}:\widehat{\pi_1(M)}\to\widehat{\pi_1(N)}$ induces the isomorphism $\overline{\phi}:\overline{\sqrt{\pi_1(M)}}\to\overline{\sqrt{\pi_1(N)}}$.

\noindent\textit{Proof of Claim 1:}
	The nilradicals $\sqrt{\pi_1(M)}$ and $\sqrt{\pi_1(N)}$ are the maximal normal free abelian subgroups of $\pi_1(M)$ and $\pi_1(N)$ respectively. Since $\widehat{\phi}(\sqrt{\pi_1(M)})\cap\pi_1(N)\subset\sqrt{\pi_1(N)}$ and $\widehat{\phi}(\overline{\sqrt{\pi_1(M)}})\subset\overline{\widehat{\phi}(\sqrt{\pi_1(M)})}$, it is obvious that $\widehat{\phi}(\overline{\sqrt{\pi_1(M)}})\subset\overline{\sqrt{\pi_1(N)}}$, hence $\overline{\phi}$ is injective. Suppose $\overline{\phi}$ is not surjective, then $\overline{\sqrt{\pi_1(N)}} \slash\overline{\phi}(\overline{\sqrt{\pi_1(M)}})$ is a finite normal abelian subgroup of $\widehat{\pi_1^{orb}(B_2)}$. While neither $\pi_1^{orb}(B_2)$ nor $\widehat{\pi_1^{orb}(B_2)}$ has any non-trivial finite normal abelian subgroup by Corollary 5.2 of \cite{BCR:2015}, which leads to the isomorphism of $\overline{\phi}$. This ends the proof of Claim 1.

\noindent\textbf{Claim 2:} Let $\pi=\pi_1(M)$, then $[\pi:C_{\pi}({\sqrt{\pi}})]$ is finite if and only if $[\widehat{\pi}:C_{\widehat{\pi}}(\overline{\sqrt{\pi}})]$ is finite.

\noindent\textit{Proof of Claim 2:}
	Suppose $[\pi:C_{\pi}(\sqrt{\pi})]=\infty$, then there would exist infinitely many $x_1,x_2,\ldots\in\pi$ such that $x_i C_\pi(\sqrt{\pi})$ are disjoint cosets of $\pi$. The cosets $x_i C_{\widehat{\pi}}(\overline{\sqrt{\pi}})$ of $\widehat{\pi}$ are disjoint, or else there exists $x_i,x_j\in\pi$ such that $x_i x_j^{-1}\in C_{\widehat{\pi}}(\overline{\sqrt{\pi}})$. However $x_i x_j^{-1}\notin C_{\pi}(\sqrt{\pi})$, which leads to contradiction. If $[\pi:C_{\pi}({\sqrt{\pi}})]<\infty$ then  $H=C_{\pi}(\sqrt{\pi})=C_{H}(\sqrt{\pi})$ is a finite index subgroup of $\pi$. So $\widehat{H}=\overline{H}$ is open in $\widehat{\pi}$. For an element $x\in \widehat{H}$, it is clear that $x$ commutes with all the elements of $\overline{\sqrt{\pi}}$, so $\widehat{H}\subseteq C_{\widehat{\pi}}(\overline{\sqrt{\pi}})$  which means that $C_{\widehat{\pi}}(\overline{\sqrt{\pi}})$ has finite index in $\widehat{\pi}$. This ends the proof of Claim 2.

Claim 2 also holds for the manifold $N$. If $M$ is geometric, the index $C_{\widehat{\pi_1(M)}}(\overline{\sqrt{\pi_1(M)}})$ in $\widehat{{\pi_1(M)}}$ is finite by Claim 2. Then  $C_{\widehat{\pi_1(N)}}(\overline{\sqrt{\pi_1(N)}})$ also has finite index in $\widehat{\pi_1(N)}$ by Claim 1. Since $C_{\widehat{\pi_1(N)}}(\overline{\sqrt{\pi_1(N)}})\cap \pi_1(N)$ is commutable with $\sqrt{\pi_1(N)}$, the index $[\pi_1(N):C_{\pi_1(N)}(\sqrt{\pi_1(N)})]<\infty$. Hence $N$ is geometric. The other direction is similar by the isomorphism $\widehat{\phi}^{-1}$.
\end{proof}

\section{Profinite completions distinguishing the geometries of Seifert manifolds}

To distinguish geometries in the sense of Thurston, we still need to clarify that the nilpotent and solvable properties are invariant under the profinite completions.

\begin{definition}
	A nilpotent group G is of \textit{nilpotent class n} if $n$ is the least integer such that $[\ldots[[x_1,x_2],x_3],\ldots,x_n]=1$ for any $x_1,x_2,\ldots,x_n\in G$.
\end{definition}

The definition is also applied to the profinite case. We just use the definition that a profinite group $\widehat{G}$ is of \textit{nilpotent class n} if $n$ is the least integer satisfying $[\ldots[[a_1,a_2],a_3],\ldots,a_n]=1\in\widehat{G}$ for any $a_1,a_2,\ldots,a_n\in \widehat{G}$. Here $a_i=(x_{ij})_{j\in I}\in\widehat{G}$ such that $x_{ij}\in G\slash N_j$ by the definition of profinite group. There is another definition that a profinite group $\Gamma$ is of \textit{nilpotent class n} if $n$ is the least integer such that $\gamma_n(\Gamma)=1$ where $\gamma_0=\Gamma$ and $\gamma_i (\Gamma)=\overline{[\Gamma,\gamma_{i-1}(\Gamma)]}$ for $i>0$, and the following remarkable theorem shows that these two definitions coincide.

\begin{theorem}\cite{NS:2006}
	Let $G$ be a finitely generated profinite group and $H$ a closed normal subgroup of $G$. Then the subgroup $[H,G]$ generated algebraically by all commutators $[h,g]$ where $h\in H$ and $g\in G$ is closed in $G$.
\end{theorem}

The following proposition should be well-known. The profinite completion maintains the nilpotent class (or solvable length) of one residually finite nilpotent (or solvable) group.

\begin{proposition}\label{prop:nil}
	A residually finite nilpotent group G is of nilpotent class $n$ if and only if $\widehat{G}$ is of nilpotent class $n$.
\end{proposition}

\begin{proof}
	If $G$ is of nilpotent class $n$, it is easy to see that the nilpotent class of $\widehat{G}$ is less than or equal to $n$. Since any $a_1,a_2,\ldots,a_n\in\widehat{G}$ with $a_i=(x_{ij})_{j\in I}$ and $x_{ij}\in G\slash N_j$, the element $[\ldots[[a_1,a_2],a_3],\ldots,a_n]\in\widehat{G}$ projects to the trivial element in each finite quotient of $G$. On the other hand, if $\widehat{G}$ is of nilpotent class $n$, for any $x_1,x_2,\ldots,x_n\in G$, we could see $x_i$ as an element of $\widehat{G}$, which means $[\ldots[[x_1,x_2],x_3],\ldots,x_n]=1\in\widehat{G}$. By the residually finite property of $G$,  $[\ldots[[x_1,x_2],x_3],\ldots,x_n] =1\in G$, which means that the nilpotent class of $G$ is less than or equal to $n$. So $G$ and $\widehat{G}$ has the same nilpotent class number.	
\end{proof}

We define a group $G$ as \textit{virtually nilpotent class} $n$ if $n$ is the least nilpotent class number of all the finite index normal nilpotent subgroups of $G$. It is well-defined in the profinite case since there is a one-to-one correspondence between the  finite index normal subgroups of $G$ and of $\widehat{G}$. The definitions and conclusion are similar for a solvable group with solvable length $n$.

\begin{corollary}
	A residually finite solvable group $G$ is of solvable length $n$ if and only if $\widehat{G}$ is of solvable length $n$.
\end{corollary}

We firstly restrict the 4-dimensional geometric manifold to be Seifert fibred, which simplifies the question to only eight geometries.

\begin{proposition}\label{prop:seifert}
	Let $M$ and $N$ be two closed orientable 4-dimensional geometric Seifert fibred manifolds. 
	If $\widehat{\pi_1(M)}\cong\widehat{\pi_1(N)}$, then $M$ and $N$ admit the same geometry.
\end{proposition}

\begin{proof}
	
As the assumption above, both $M$ and $N$ admit  $T^2$-bundle structures over closed 2-orbifolds.  The base orbifolds could be spherical, euclidean or hyperbolic. For the spherical case, the geometries are $\mathbb{S}^2\times\mathbb{E}^2$ and $\mathbb{S}^3\times\mathbb{E}$, which are firstly distinguished since the fundamental groups of corresponding manifolds are virtually $\mathbb{Z}^2$ or $\mathbb{Z}$. Then $\widehat{{\pi_1(M)}}$ is virtually $\widehat{ \mathbb{Z}^2}$ or $\widehat{ \mathbb{Z}}$ if and only if $\widehat{\pi_1(N)}$ is virtually $\widehat{ \mathbb{Z}^2}$ or $\widehat{ \mathbb{Z}}$ respectively, which means that $M$ admits geometry $\mathbb{S}^2\times\mathbb{E}^2$ or $\mathbb{S}^3\times\mathbb{E}$ if and only if $N$ admits same geometry respectively.
	
When the orbifold is euclidean, it is clear that $M$ admits $\mathbb{E}^4$-geometry if and only if $\pi_1(M)$ is virtually $\mathbb{Z}^4$. For the geometries $\mathbb{N}il^3\times\mathbb{E}$ and $\mathbb{N}il^4$, the nilradical $\sqrt{\pi_1(M)}$ is maximal nilpotent with finite index and has nilpotent class 2 and 3 respectively by Chapter 7 of \cite{Hi:2002}. The manifold $M$ admits $\mathbb{S}ol^3\times\mathbb{E}$-geometry if and only if $\pi_1(M)$ is virtually solvable but not virtually nilpotent. By Proposition \ref{prop:nil} and Proposition \ref{profinite1-1}, $M$ is of geometry $\mathbb{E}^4$, $\mathbb{N}il^3\times\mathbb{E}$, $\mathbb{N}il^4$ or $\mathbb{S}ol^3\times\mathbb{Z}$ geometry if and only if $\widehat{{\pi_1(M)}}$ is virtually $\widehat{\mathbb{Z}^4}$, virtually nilpotent class 2, virtually nilpotent class 3 or virtually solvable but not virtually nilpotent. By the isomorphism between $\widehat{\pi_1(M)}$ and $\widehat{\pi_1(N)}$, $M$ admits geometry $\mathbb{E}^4$, $\mathbb{N}il^3\times\mathbb{E}$, $\mathbb{N}il^4$ or $\mathbb{S}ol^3\times\mathbb{Z}$ if and only if $N$ admits the same geometry respectively.
	
Now we focus on the situation when the base orbifolds are hyperbolic. By Theorem  \ref{thm:h2*e2} and Theorem \ref{thm:sl*e}, $M$ admits geometry $\mathbb{H}^2\times\mathbb{E}^2$ if and only if $e^{\mathbb{Q}}(\pi_1(M))=0$, and $M$ admits geometry $\widetilde{\mathbb{S}L_2}\times\mathbb{E}$ if and only if $e^{\mathbb{Q}}(\pi_1(M))\neq0$. The condition $e^{\mathbb{Q}}(\pi_1(M))=0$  means that $M$ not only is a Seifert fibred manifold over a hyperbolic base orbifold, but also  admits a hyperbolic orbifold bundle  over $T^2$. So we could know that it is equivalent to that there exists some finite cover $M'$ of $M$ associating to a surface $\Sigma$ such that the short exact sequence
	\begin{align}\label{1}
		1\to\mathbb{Z}^2\to\pi_1(M')\to\pi_1(\Sigma)\to 1
	\end{align}
splits.

\noindent\textbf{Claim:} The short exact sequence (\ref{1}) splits if and only if its profinite completion (\ref{2}) splits.
    \begin{align}\label{2}
		1\to\widehat{\mathbb{Z}^2}\to\widehat{\pi_1(M')}\to\widehat{\pi_1(\Sigma)}\to 1
	\end{align}

\noindent\textit{Proof of Claim:} This claim is similar to the situation in \cite{Wil:2017}. It is clear that if (\ref{1}) splits, then (\ref{2}) naturally splits. Now we suppose (\ref{1}) does not split and consider another non-splitting short exact sequence

	\begin{center}
		$1\to\mathbb{Z}^2\xrightarrow{A}\mathbb{Z}^2\to G\to1$
	\end{center}
where $A$ is an integer matrix and the absolute value of its determinant is larger than 1. Then $G$ is a finite abelian group such that $G\cong\mathbb{Z}\slash a_1\mathbb{Z}\times\mathbb{Z}\slash a_2\mathbb{Z}$ with  $a_1,a_2\in\mathbb{Z}^+$. Since we assume that (\ref{1}) does not split, it induces a long exact sequence in cohomology
		
	\centerline{\xymatrix{\ldots\ar[r]&H^2(\pi_1(\Sigma),\mathbb{Z}^2)\ar[r]&H^2(\pi_1(\Sigma),G)\ar[r]&H^1(\pi_1(\Sigma),\mathbb{Z}^2)\ar[r]&\ldots.}}

The colomology group $H^1(\pi_1(\Sigma),\mathbb{Z}^2)$ is torsion-free, but 	$H^2(\pi_1(\Sigma),G)$ is finite which leads to the surjection of $H^2(\pi_1(\Sigma),\mathbb{Z}^2)\to H^2(\pi_1(\Sigma),G)$. Since (\ref{1}) does not split, there is a corresponding non-trivial element in $H^1(\pi_1(\Sigma),\mathbb{Z}^2)$ which projects to a non-trivial element in $H^2(\pi_1(\Sigma),G)$ for some finite quotient $G$ of $\mathbb{Z}^2$. Then there is a non-splitting short exact sequence $1\to G\to G^{\ast}\to\pi_1(\Sigma)\to 1$ which  corresponds to the element in $H^2(\pi_1(\Sigma),G)$, where $G^{\ast}$ is the extension of $G$ by $\pi_1(\Sigma)$. By the goodness of $\pi_1(\Sigma)$ such that $H^2(\pi_1(\Sigma),G)\cong H^2(\widehat{\pi_1(\Sigma)},G)$, we know that $1\to G\to \widehat{G^{\ast}}\to\widehat{\pi_1(\Sigma)}\to 1$ does not split, so the inverse limit (\ref{2}) does not split. This ends the proof of Claim.

Hence the geometry of $M$ is $\mathbb{H}^2\times\mathbb{E}^2$ if and only if for any finite index normal subgroup of $\widehat{\pi_1(M)}$, the induced short exact sequence splits. It leads to our conclusion that $M$ and $N$ admit the same geometry.
\end{proof}

From Theorem \ref{thm:geo} and Proposition \ref{prop:seifert}, we have the following corollary:

\begin{corollary}
	Let $M$ be a closed 4-dimensional Seifert fibred manifold with fundamental group $\pi_1(M)$. Then whether $M$ is geometric is detected by the profinite completion $\widehat{\pi_1(M)}$. If $M$ is geometric, its geometry type can also be distinguished by $\widehat{\pi_1(M)}$.

\end{corollary}

\section{Profinite completions distinguishing the geometries of solvable Lie types}

For a closed  $\mathbb{S}ol^4_1$-manifold $M$, it is homeomorphic to a mapping torus of a $\mathbb{N}il^3$-manifold if and only if $M$ is orientable by Theorem 8.9 of \cite{Hi:2002}. Since $\sqrt{\pi_1(M)}\cong \Gamma_q$ for some $q>1$ as $\Gamma_q$ defined in (\ref{3}), let $\varphi\in Aut(\Gamma_q)$ be an automorphism of $\Gamma_q$, sending $x$ to $x^a y^b z^m$ and $y$ to $x^c y^d z^n$ for $C=\begin{pmatrix}
a&b\\c&d
\end{pmatrix}\in GL(2,\mathbb{Z})$. The matrix $C$ could be seen as the automorphism on $\Gamma_q\slash I(\Gamma_q)\cong\mathbb{Z}^2$, where $I(\Gamma_q)\cong\mathbb{Z}$ denotes the subgroup generated by $z$. So $\varphi$ could be represented as a tuple $(C,\mu)\in GL(2,\mathbb{Z})\times\mathbb{Z}^2$ where $\mu=(m,n)\in\mathbb{Z}^2$. On the other hand, every $\varphi=(C,\mu)$ represents an automorphism of $\Gamma_q$, and $\Gamma_q\rtimes_{\varphi}\mathbb{Z}$ is the fundamental group of a mapping torus of a $\mathbb{N}il^3$-manifold. Such a mapping torus could admit geometry $\mathbb{N}il^3\times \mathbb{E}$, $\mathbb{N}il^4$ or $\mathbb{S}ol^4_1$ by Theorem 8.9 of \cite{Hi:2002}.

 Let $N_A=N_B\cong\mathbb{Z}^3$ and $H\cong\mathbb{Z}$ be free abelian groups, and define the corresponding semi-direct products  $G_A=N_A\rtimes_A H$ and $G_B=N_B\rtimes_B H$, where $A,B\in SL(3,\mathbb{Z})$. We could conclude the condition for such two groups being isomorphic which is similar to Lemma 2.1 of \cite{Nery:2020}.

\begin{lemma}
 Let $A,\ B\in SL(3,\mathbb{Z})$ such that both $A$ and $B$ have at most one eigenvalue equal to 1. If $G_A$ and $G_B$ are isomorphic, then the isomorphism sends $N_A$ to $N_B$. What's more, $G_A$ is isomorphic to $G_B$ if and only if $A$ and $B$ (or $B^{-1}$) are conjugate in $GL(3,\mathbb{Z})$.
\end{lemma}

\begin{proof}
	
	Let $\phi:G_A\to G_B$ denote the isomorphism. As assumption, neither $A-I$ nor $B-I$ is nilpotent since both $A$ and $B$ contain at most one eigenvalue which is equal to 1. Then both $G_A$ and $G_B$ are solvable but not nilpotent. It should be noted that $N_A$ is the unique maximal abelian normal subgroup of $G_A$ since $\left\langle N_A,t\right\rangle $  is not abelian for any non-trivial $t\in H$, and similarly $N_B$ is the  unique maximal abelian normal subgroup of $G_B$. Hence the image of $N_A$ is a subset of $N_B$. The isomorphism $\phi:G_A\to G_B$ induces the isomorphism between $N_A$ and $N_B$ since $H\cong \mathbb{Z}$ does not have any finite subgroup. The second part is directly from Corollary 2.2 and Proposition 2.5 of \cite{GZ:2011} for $G_A$ and $G_B$ to be isomorphic.
\end{proof}

The similar proof can apply to the profinite completions. We denote $A$ and $B$ as  the semi-direct actions on $\widehat{\mathbb{Z}}^3$  induced by the embedding  $ GL(3,\mathbb{Z})\to GL(3,\widehat{\mathbb{Z}})$. Then we have the following corollary by Corollary 2.2 and Proposition 2.5 of \cite{GZ:2011}.
\begin{corollary}\label{cor:semidirect}
 Let $A,\ B\in SL(3,\mathbb{Z})$ such that both $A$ and $B$ have at most one eigenvalue equal to 1. If  $\widehat{G_A}$ and $\widehat{G_B}$ are isomorphic, then the isomorphism sends $\widehat{N_A}$ to $\widehat{N_B}$. What's more, $\widehat{G_A}$  is isomorphic to $\widehat{G_B}$ if and only if $A$ and $B$ (or $B^{-1}$) are conjugate in $GL(3,\widehat{ \mathbb{Z}})$.
\end{corollary}
We should remark that it is similar to the  profinite flexibility exhibited by Funar \cite{Funar}. Two matrices $A,\ B\in SL(3,\mathbb{Z})$ are conjugate in $GL(3,\widehat{ \mathbb{Z}})$ if and only if $A$ and $B$ project to conjugate matrices in $GL(3,\mathbb{Z}\slash q\mathbb{Z})$ for every integer $q>1$.
Now we can deal with geometries $\mathbb{S}ol^4_0$, $\mathbb{S}ol^4_1$ and $\mathbb{S}ol^4_{m,n}$. For a closed orientable manifold with one of these geometries, its fundamental group is isomorphic to a semi-direct product $G\rtimes\mathbb{Z}$ where $G\cong\mathbb{Z}^3$ or $\Gamma_q$. The profinite completion maintains some features of groups for being nilpotent or solvable as we proved in Proposition \ref{prop:nil}, which we could use to distinguish the geometries of solvable Lie type.

\begin{proposition}\label{prop:sol}
	Let $M$, $N$ be two closed orientable 4-dimensional geometric manifolds, and suppose $M$ and $N$ admit distinct geometries among $\mathbb{S}ol^4_0$, $\mathbb{S}ol^4_{m,n}$ and $\mathbb{S}ol^4_1$. Then $\widehat{\pi_1(M)}\ncong\widehat{\pi_1(N)}$.
\end{proposition}
\begin{proof}
    As assumption, both $M$ and $N$ are 4-dimensional geometric manifolds with distinct geometries. Suppose there exists an isomorphism  $\widehat{\phi}:\widehat{\pi_1(M)}\to\widehat{\pi_1({N})}$. Firstly, assume that $M$  admits geometry $\mathbb{S}ol^4_0$ or $\mathbb{S}ol^4_{m,n}$, and $N$ is a $\mathbb{S}ol^4_1$-manifold. The nilradical $\sqrt{\pi_1(M)}$ is isomorphic to $\mathbb{Z}^3$ which means that  $\pi_1(M)\cong\mathbb{Z}^3\rtimes_A\mathbb{Z}$ where $A\in SL(3,\mathbb{Z})$. So $\pi_1(M)$ is metabelian with solvable length $2$. On the other hand, $N$ is the mapping torus of a $\mathbb{N}il^3$-manifold with  $\pi_1(N)\cong\Gamma_q\rtimes_{\varphi}\mathbb{Z}$ for some $q>1$, where $\varphi=(C,\mu)\in Aut(\Gamma_q)$. It's clear that $\pi_1(N)$ is solvable with composition series $1\triangleleft\mathbb{Z}^2\triangleleft\Gamma_q\triangleleft\pi_1(M)$, which means that its solvable length is at most 3. The commutator subgroup $G=[\pi_1(N),\pi_1(N)]\triangleleft\Gamma_q$, for $\pi_1(N)\slash\Gamma_q\cong\mathbb{Z}$ is abelian. By Theorem 8.7 of \cite{Hi:2002}, $C\in GL(2,\mathbb{Z})$ has infinite order with distinct eigenvalues which are not equal to $\pm 1$, hence $\beta_1(N)=1$, which means that $G$ is a finite index normal subgroup of $\Gamma_q$. Since $\Gamma_q$ is the fundamental group of some $\mathbb{N}il^3$-manifold $S$, $G\cong\pi_1(S')$ where $S'$ is some finite-sheeted cover of $S$. If $G$ is abelian, by the classification of abelian 3-manifold groups in \cite{ASW:2012}, $G$ could only be $\mathbb{Z}^3$, $\mathbb{Z}$, $\mathbb{Z}\slash n\mathbb{Z}$ or $\mathbb{Z}\times\mathbb{Z}\slash 2 \mathbb{Z}$, which means $S'$ admits geometry $\mathbb{E}^3$, $\mathbb{S}^2\times\mathbb{E}$, $\mathbb{S}^3$ or $\mathbb{S}^1\times\mathbb{S}^2$ respectively. The situation is contradict to the fact that $S$ admits geometry $\mathbb{N}il^3$, hence $\pi_1(N)$ is not metabelian, such that it has solvable length $3$. By the invariance of solvable length of discrete group under profinite completion, $\widehat{\pi_1(M)}$ and $\widehat{\pi_1(N)}$ have solvable length 2 and 3 respectively, which leads to the contradiction.

    Now we adjust the assumption such that neither $M$ nor $N$ admits geometry $\mathbb{S}ol^4_1$.
    Hence both nilradicals $\sqrt{\pi_1(M)}$ and $\sqrt{\pi_1(N)}$ are isomorphic to $\mathbb{Z}^3$. So we could assume that $\pi_1(M)$ and $\pi_1(N)$ are semi-direct products $\mathbb{Z}^3\rtimes_A\mathbb{Z}$ and $\mathbb{Z}^3\rtimes_B\mathbb{Z}$ respectively where $A,B\in SL(3,\mathbb{Z})$. Both $A$ and $B$ have at most one eigenvalue equal to 1. By the assumption, the isomorphism $\widehat{\phi}$ sends $\overline{\sqrt{\pi_1(M)}}$ to $\overline{\sqrt{\pi_1(N)}}$ naturally, since $\widehat{\phi}(\sqrt{\pi_1(M)})\cap\pi_1(N)\subset\sqrt{\pi_1(N)}$, and $\widehat{\phi}(\overline{\sqrt{\pi_1(M)}})\subset\overline{\widehat{\phi}(\sqrt{\pi_1(M)})}$. Then $\widehat{\pi_1({M})}\cong\widehat{\mathbb{Z}}^3\rtimes_{\widehat{A}} \widehat{\mathbb{Z}}$ is isomorphic to $\widehat{\mathbb{Z}}^3\rtimes_{\widehat{B}} \widehat{\mathbb{Z}}$ if and only if $A$ is conjugate to $B$ (or $B^{-1}$) in $GL(3,\widehat{\mathbb{Z}})$ by
    Corollary \ref{cor:semidirect}. We could simply suppose $A$ and $B$ are conjugate in $GL(3,\widehat{\mathbb{Z}})$. It is clear that $GL(3,\widehat{\mathbb{Z}})$ surjects onto $GL(3,\mathbb{Z}\slash q \mathbb{Z})$ for each integer $q>1$, which means that the images of $A$ and $B$ are conjugate in $GL(3,\mathbb{Z}\slash q \mathbb{Z})$. Then it is trivial that $tr(A)\equiv tr(B)\ (mod\ q)$ and $tr(A^{-1})\equiv tr(B^{-1})\ (mod\ q)$ for every $q>1$, such that $A$ and $B$ admit same characteristic polynomials $f_A(x)=f_B(x)=x^3-tr(A)x^2+tr(A^{-1})x-1$. However the geometries of $M$ and $N$ depend on the roots of $f_A(x)$ by the introduction in Section 2, which means that $M$ and $N$ are of same geometry. If  $f_A(x)$ has a couple of conjugate complex roots, then $M$ and $N$ admit geometry $\mathbb{S}ol^4_0$. When $f_A(x)$ has three distinct roots, $M$ and $N$ are admitting geometry $\mathbb{S}ol^4_{m,n}$ where $(m,n)=(tr(A), tr(A^{-1}))$ if all three roots are positive, or $(m,n)=(tr(A^2), tr(A^{-2}))$ if not.
\end{proof}

By Remark 3.1.A of Gromov \cite{Gromov} that if $G$ is a subgroup of some word hyperbolic group, then $G$ contains the free group $F_2$ or $G$ is a finite extension of some cyclic group. So the geometries with hyperbolic factor could be distinguished from the solvable geometries.

\begin{proposition}\label{prop:hyp&sol}
	Let $M$, $N$ be two closed orientable 4-dimensional geometric manifolds. Suppose $M$ admits one of geometries including $\mathbb{H}^{3}\times \mathbb{E}$, $ \widetilde{\mathbb{S}L_2}\times \mathbb{E}$, $ \mathbb{H}^{2}\times \mathbb{E}^{2}$, $\mathbb{H}^{2}\times \mathbb{H}^{2}$, $ \mathbb{H}^{4}$, $\mathbb{H}^{2}_{\mathbb{C}}$, and $\mathbb{H}^2\times\mathbb{S}^2$, and $N$ admits one of solvable Lie type geometries. Then $\widehat{\pi_1(M)}\ncong\widehat{\pi_1(N)}$.
\end{proposition}

\begin{proof}
	  If $M$ admits geometry $\mathbb{H}^{3}\times \mathbb{E}$, $ \widetilde{\mathbb{S}L_2}\times \mathbb{E}$, $ \mathbb{H}^{2}\times \mathbb{E}^{2}$, $ \mathbb{H}^{4}$, $\mathbb{H}^{2}_{\mathbb{C}}$, or $M$ is a reducible $\mathbb{H}^{2}\times \mathbb{H}^{2}$-manifold, we could know that $\pi_1(M)$ contains the free group $F_2$ by  Remark 3.1.A of Gromov \cite{Gromov}, so $\pi_1(M)$ is not solvable. On the other hand, $\pi_1(N)$ is a virtually polycyclic and solvable group with $h(\pi_1(N))=4$ which means that $\pi_1(N)$ is not a finite extension of cyclic group. Since profinite completion maintains the property of being solvable, it is easy to see that $\widehat{\pi_1(M)}\ncong\widehat{\pi_1({N})}$. 
	  
	  If $M$ is an irreducible $\mathbb{H}^2\times\mathbb{H}^2$-manifold,  $\pi_1(M)$ is an arithmetic subgroup of $Isom(\mathbb{H}^2\times\mathbb{H}^2)$ with no non-trivial normal subgroup of infinite index by Margulis Normal subgroup Theorem. Then $\beta_1 (\pi_1(M))=0$. On the other hand,   $\beta_1(\pi_1(N))\geq 1$ since $N$   admits some solvable Lie geometry. By the profinite invariance of $\beta_1$ of  discrete groups \cite{BCR:2015}, $\widehat{\pi_1(M)}\ncong\widehat{\pi_1({N})}$.
\end{proof}

\section{Profinite completions distinguishing geometries with hyperbolic factor}

In the last part of this paper, we prove that the profinite completion could distinguish geometries except for three cases eliminated in Theorem \ref{thm:1}.

\begin{proposition}
	Let $M$, $N$ be two closed orientable 4-dimensional manifolds. If $M$ admits geometry $\mathbb{H}^2\times\mathbb{H}^2$, and $N$ admits one of geometries $\mathbb{H}^2\times\mathbb{S}^2$, $\mathbb{H}^3\times\mathbb{E}$, $\mathbb{H}^2\times\mathbb{E}^2$ and $ \widetilde{\mathbb{S}L_2}\times \mathbb{E}$, then $\widehat{\pi_1(M)}\ncong\widehat{\pi_1(N)}$.
\end{proposition}

\begin{proof}
	Suppose there is an isomorphism $\widehat{\phi}:\widehat{{\pi_1(M)}}\to \widehat{{\pi_1(N)}}$. Firstly, we assume that $M$ is an irreducible $\mathbb{H}^2\times\mathbb{H}^2$-manifold, which means that it can not be covered by the product of two closed orientable surfaces by the definition in Chapter 14 of Hillman's \cite{Hi:2002}, then $\pi_1(M)$ is an arithmetic subgroup of $Isom(\mathbb{H}^2\times\mathbb{H}^2)$. By the Margulis Normal Subgroups Theorem \cite{Ma:1991}, $\pi_1(M)$ has no non-trivial normal subgroups of infinite index, which means $\beta_1(M)=0$, and $\beta_1(G)=0$ for any finite index normal subgroup $G\triangleleft \pi_1(M)$. Meanwhile $N$ is of $\mathbb{H}^3\times\mathbb{E}$, $\mathbb{H}^2\times\mathbb{E}^2$, $\mathbb{H}^2\times\mathbb{S}^2$ or $\widetilde{\mathbb{S}L_2}\times\mathbb{E}$ geometry, which means that there is a finite-sheeted cover $N'$ of $N$ such that $\beta_1(\pi_1(N'))\geq 1$. Hence there exists some finite-sheeted cover $M'$ of $M$ such that $\widehat{{\pi_1(M')}}\cong\widehat{{\pi_1(N')}}$. However, by the profinite invariance of $\beta_1$ in \cite{BCR:2015}, it is contradict to that $\beta_1(\pi_1(M'))=0$. Hence $\widehat{\pi_1(M)}\ncong\widehat{\pi_1(N)}$.

Now suppose $M$ is a reducible $\mathbb{H}^2\times\mathbb{H}^2$-manifold.  For a finitely generated residually finite group $\Omega$, its $L^2$-Betti number, by L\"{u}ck's Approximation Theorem in \cite{luck:1994}, is
$$b_1^{(2)}(\Omega)=\lim\limits_{d\to\infty}\frac{\beta_1(\Omega_d)}{[\Omega : \Omega_d]},$$ 
where $\Omega_d$ is the intersection of all normal subgroups of index at most $d$ in $\Omega$. If $\Omega$ is a surface group, then $b_1^{(2)}(\Omega)=-\chi(\Omega)$ by \cite[Propositon 3.5]{BCR:2015}. By assumption, $M$ could be finitely covered by the product of two surfaces, which means that there are two surface groups $\Omega$ and $\Lambda$ such that $\Omega\times\Lambda$ is a normal subgroup of $\pi_1(N)$ with finite index $m$. Then $$b_1^{(2)}(\pi_1(M))=\frac{1}{m}\lim\limits_{c,d\to\infty}\frac{\beta_1(\Omega_c\times\Lambda_d)}{[\Omega\times\Lambda : \Omega_c\times\Lambda_d]}=\frac{1}{m}\lim\limits_{c,d\to\infty}\frac{\beta_1(\Omega_c)+\beta_1(\Lambda_d)}{[\Omega:\Omega_c][\Lambda:\Lambda_d]}=0.$$ On the other hand, if $N$ admits geometry $\mathbb{H}^2\times\mathbb{S}^2$, its fundamental group $\pi_1(N)$ is virtually a surface group $\pi_1(\Sigma)$. Then  $b_1^{(2)}(\pi_1(N))=-\chi(\pi_1(\Sigma))\slash n$  where $n$ is the index of $\pi_1(\Sigma)$ in $\pi_1(N)$. Then $b_1^{(2)}(\pi_1(N))$ is not zero since $\Sigma$ admits $\mathbb{H}^2$-geometry.  By the invariance of $L^2$-Betti number under profinite completion \cite[Proposition 3.2]{BCR:2015}, $\widehat{\pi_1(M)}\ncong\widehat{\pi_1(N)}$.
 
For $M$ is still a reducible $\mathbb{H}^2\times\mathbb{H}^2$-manifold, we consider the finite $m$-sheeted cover $M'$ of $M$ such that $\pi_1(M')\cong\Omega\times\Lambda$. Suppose $N$ admits one of geometries among $\mathbb{H}^3\times\mathbb{E}$, $\mathbb{H}^2\times\mathbb{E}^2$ and $ \widetilde{\mathbb{S}L_2}\times \mathbb{E}$. Since $[\widehat{\pi_1(M)}:\widehat{\pi_1(M')}]=[\pi_1(M):\pi_1(M')]=m$, there is also an $m$-sheeted cover $N'$ of $N$ such that $\widehat{\pi_1(M')}\cong\widehat{\pi_1(N')}$. By the goodness of $\pi_1(M')$ and $\pi_1(N')$, 	
$$H^{\ast}(\pi_1(M'), \mathbb{Z}/q\mathbb{Z})\cong H^{\ast}(\widehat{\pi_1(M')}, \mathbb{Z}/q\mathbb{Z})\cong H^{\ast}(\widehat{\pi_1(N')}, \mathbb{Z}/q\mathbb{Z}) \cong H^{\ast}(\pi_1(N'), \mathbb{Z}/q\mathbb{Z}) $$ for any prime $p\in \mathbb{Z}$. For $M'$ and $N'$ are aspherical finite complexes, $\chi(\pi_1(M'))=\chi(M')=\sum_{i=0}^{4}(-1)^i\beta_i(M')$ which is the alternating sum of Betti numbers.
Choose $q$ to be coprime to all the order of torsion  elements in homology groups of $\pi_1(M')$ and $\pi_1(N')$. Then the euler characteristic number is invariant when passing to the $\mathbb{Z}\slash q\mathbb{Z}$-coefficient homology group.  By the summary in Chapter 9 of Hillman's \cite{Hi:2002}, $\chi(N')=0$ since $N'$ admits  $\mathbb{H}^3\times\mathbb{E}$, $\mathbb{H}^2\times\mathbb{E}^2$ or $\widetilde{\mathbb{S}L_2}\times\mathbb{E}$ geometry, and $\chi(M')\neq0$ since $M'$ is a reducible $\mathbb{H}^2\times\mathbb{H}^2$-manifold, which leads to the contradiction. Hence $\widehat{{\pi_1(M)}}\ncong\widehat{{\pi_1(N)}}$.
\end{proof}

\begin{proposition}\label{prop:hyp1}
	Let $M$, $N$ be two closed orientable 4-dimensional manifolds. If $M$ admits geometry $\mathbb{H}^2_{\mathbb{C}}$ or $\mathbb{H}^4$, and $N$ admits one of geometries $\mathbb{H}^2\times\mathbb{S}^2$, $\mathbb{H}^3\times\mathbb{E}$, $\mathbb{H}^2\times\mathbb{E}^2$ and $ \widetilde{\mathbb{S}L_2}\times \mathbb{E}$, then $\widehat{\pi_1(M)}\ncong\widehat{\pi_1(N)}$.
\end{proposition}

\begin{proof}
	
Suppose there is an isomorphism $\widehat{\phi}:\widehat{{\pi_1(M)}}\to \widehat{{\pi_1(N)}}$, where $M$ is as assumption and $N$ admits $\mathbb{H}^3\times\mathbb{E}$-geometry. Then $\pi_1(N)$ has a finite index normal subgroup $\pi_1(S)\times\mathbb{Z}$ where $S$ is a closed 3-dimensional $\mathbb{H}^3$-manifold. The profinite completion $\widehat{\pi_1(S)\times\mathbb{Z}}$ is naturally $\widehat{\pi_1(S)}\times\widehat{\mathbb{Z}}$ by the goodness of $\pi_1(S)$ and $\mathbb{Z}$. Accordingly, there is a finite-sheeted cover $M'$ of $M$ such that $\widehat{\phi}$ induces the isomorphism $\widehat{\phi}:\widehat{{\pi_1(M')}}\to\widehat{\pi_1(S)}\times\widehat{ \mathbb{Z}}$. Consider the homomorphisms 
$$\widehat{{\pi_1(M')}}\xrightarrow{\widehat{\phi}}\widehat{\pi_1(S)}\times\widehat{\mathbb{Z}}\xrightarrow{p}\widehat{\pi_1(S)}$$
where $p$ is the canonical projection.

\noindent\textbf{Claim:} $p\circ\widehat{\phi}:\pi_1(M')\to\widehat{\pi_1(S)}$ is injective.

\noindent\textit{Proof of Claim:} Suppose there exists non-trivial element $x\in\pi_1(M')$ such that $p\circ\widehat{\phi}(x)=1\in\widehat{\pi_1(S)}$. Then $\widehat{\phi}(x)=(1,x'')\in\widehat{\pi_1(S)}\times\widehat{\mathbb{Z}}$ for some non-trivial $x''\in\mathbb{Z}$. For any element $y\in\pi_1(M)$ such that $\widehat{\phi}(y)=(y',y'')\in\widehat{S}\times\widehat{\mathbb{Z}}$, we have $\widehat{\phi}(xy)=\widehat{\phi}(yx)$. Hence $xy=yx\in\pi_1(M')$, which means that $x$ is in the centre of $\pi_1(M')$. However, $\pi_1(M')$ has no centre, which finishes the proof of Claim.

Now $\pi_1(M')<\widehat{\pi_1(S)}$, then the profinite completion $\widehat{\pi_1(M')}\leq\widehat{\pi_1(S)}$, which means that  $\widehat{\pi_1(S)}$ contains $\widehat{\mathbb{Z}}^2$ as a subgroup. However, it is contradict to the $\mathbb{H}^3$-geometry of $S$ by Theorem A of \cite{Wil:2017}. Hence $\widehat{\pi_1({M})}\ncong\widehat{\pi_1(N)}$.

The case  for $N$ admitting geometry $\mathbb{H}^2\times\mathbb{E}^2$ or $\widetilde{\mathbb{S}L}\times\mathbb{E}$ is similar. As for $\mathbb{H}^2\times\mathbb{S}^2$-geometry, $\pi_1(N)$ is virtually a surface group $\pi_1(\Sigma_g)$ for some $g>1$. By Theorem 1.1 of \cite{BCR:2015} that $\pi_1(N)$ is profinitely rigid among all the lattices of connected Lie group, we could know that $\widehat{\pi_1(M)}\ncong\widehat{\pi_1(N)}$.	
\end{proof}

Finally we distinguish geometries $\mathbb{H}^2\times\mathbb{S}^2$, $\mathbb{H}^3\times\mathbb{E}$,  $\mathbb{H}^2\times\mathbb{E}^2$ and $\widetilde{\mathbb{S}L_2}\times\mathbb{E}$ from each other. The last two are already dealt with in  Proposition \ref{thm:geo} for they could both be realized as Seifert fibred spaces.
\begin{proposition}
	Let $M$, $N$ be two closed orientable 4-dimensional manifolds admitting distinct geometries among $\mathbb{H}^2\times\mathbb{S}^2$, $\mathbb{H}^3\times\mathbb{E}$,  $\mathbb{H}^2\times\mathbb{E}^2$ and $\widetilde{\mathbb{S}L_2}\times\mathbb{E}$, then $\widehat{\pi_1(M)}\ncong\widehat{\pi_1(N)}$.
\end{proposition}

\begin{proof}
Firstly, we suppose that there is an isomorphism between two profinite completions $\widehat{\phi}:\widehat{\pi_1(M)}\to \widehat{\pi_1(N)}$, and assume $M$ is a $\mathbb{H}^2\times\mathbb{S}^2$-manifold while $N$ admits one of last three geometries. It is clear that $\pi_1(M)$ is virtually a surface group $\pi_1(\Sigma_g)$ with genus $g\geq2$, and $\pi_1(\Sigma_g)$ doesn't contain non-trivial centre. On the other hand, there exists a finite index normal subgroup $\Lambda\triangleleft\pi_1(N)$ such that $\Lambda$ has non-trivial centre since its geometry contains Euclidean factor. Then it is contract to the assumption of isomorphism which is similar to Proposition \ref{prop:hyp1}, hence $\widehat{\pi_1(M)}\ncong\widehat{\pi_1(N)}$.

Now assume $M$ is of geometry $\mathbb{H}^3\times\mathbb{E}$, and $N$ is of geometry $\mathbb{H}^2\times\mathbb{E}^2$ or $\widetilde{\mathbb{S}L_2}\times\mathbb{E}$.
Then there is a finite-sheeted cover $M'$ of $M$ such that $\pi_1(M')\cong\pi_1(S)\times\mathbb{Z}$ and $S$ is a $\mathbb{H}^3$-manifold. Accordingly, there is a finite-sheeted cover $N'$ of $N$ such that $\widehat{\pi_1(M')}\cong\widehat{\pi_1(N')}$ by the isomorphism of $\widehat{\phi}$. The nilradical  $\sqrt{\pi_1(N')}\cong\mathbb{Z}^2$ since $N'$ is a Seifert fibred 4-manifold, which means that $\widehat{\phi}^{-1}(\overline{\sqrt{\pi_1(N')}})\cong\widehat{ \mathbb{Z}}^2$. Like Proposition \ref{prop:hyp1}, we use $p:\widehat{\pi_1(S)}\times\widehat{\mathbb{Z}}\to\widehat{\pi_1(S)}$ to denote the canonical projection. Then $p\circ\widehat{\phi}^{-1}(\overline{\sqrt{\pi_1(N')}})$ could only be isomorphic to $\widehat{\mathbb{Z}}^2$, $\prod_{p\in\pi}\mathbb{Z}_p$  or $\widehat{\mathbb{Z}}\times\prod_{p\in\pi}\mathbb{Z}_p$, where $\pi$ is some non-trivial collection of prime numbers and $\mathbb{Z}_p$ is $p$-adic integer since $\widehat{\mathbb{Z}}\cong\prod_{p\ prime}\mathbb{Z}_p$. However, these three situations are all impossible, which means that $\widehat{\pi_1(M')}\ncong\widehat{\pi_1(N')}$. If $p\circ\widehat{\phi}^{-1}(\overline{\sqrt{\pi_1(N')}})\cong\widehat{\mathbb{Z}}^2$, then $\widehat{\pi_1(S)}$ contains $\widehat{\mathbb{Z}}^2$ as a subgroup, which means that $S$ is not hyperbolic by Theorem A of \cite{Wil:2017}. If $p\circ\widehat{\phi}^{-1}(\overline{\sqrt{\pi_1(N')}})\cong\widehat{\mathbb{Z}}\times\prod_{p\in\pi}\mathbb{Z}_p$, then $\widehat{\pi_1(S)}$ contains $\widehat{\mathbb{Z}}$ as a procyclic normal subgroup, which means that $S$ is Seifert fibred by Theorem B of \cite{Wil:2017}. If $p\circ\widehat{\phi}^{-1}(\overline{\sqrt{\pi_1(N')}})\cong\prod_{p\in\pi}\mathbb{Z}_p$,  then $\widehat{\phi}^{-1}$ induces the injection $\widehat{ \mathbb{Z}}^2\to\prod_{p\in\pi}\mathbb{Z}_p\times\widehat{\mathbb{Z}}$, which is not possible. Hence $\widehat{\pi_1(M)}\ncong\widehat{\pi_1(N)}$. 
\end{proof}

\begin{proof}[\textbf{Proof of Theorem \ref{thm:1}}]
	As assumption, both $M$ and $N$ are closed orientable 4-dimensional manifolds with infinite fundamental groups and distinct geometries. By Proposition \ref{prop:hyp&sol}, we could know that $M$ admits solvable Lie geometries if and only if $N$ admits solvable Lie geometries. And all the solvable geometries including infinite types in $\mathbb{S}ol^4_{m,n}$ are all distinguished from each other by Proposition \ref{prop:seifert} and Proposition \ref{prop:sol}. For the geometries with hyperbolic factor, they can be distinguished by their profinite completions by the discussion in Section 8 except for the pairs listed in the theorem.
\end{proof}

\begin{proof}[\textbf{Proof of Corollary \ref{Cor:3}}]
	 It is easy to describe the geometry of a closed orientable 4-manifold $M$ by the profinite completion of its fundamental group $\widehat{\pi_1(M)}$ when $M$ is  solvable, Seifert fibred, or has a spherical factor. The only part left is when $M$ admits geometry $\mathbb{H}^3\times\mathbb{E}$. Then there is a finite index normal subgroup $\pi_1(S)\times\mathbb{Z}$ of $\pi_1(M)$ such that $S$ is a closed orientable $\mathbb{H}^3$-manifold. Then $\widehat{\pi_1(S)}\times\widehat{\mathbb{Z}}$ is a finite index normal subgroup of $\widehat{\pi_(M)}$. By the similar method in Proposition \ref{prop:hyp1}, the subgroup $\widehat{\mathbb{Z}}^3$  of $\widehat{\pi_1(M)}$  could project to some subgroup $G$ of $\widehat{\pi_1(S)}$, and $G$ contains $\widehat{\mathbb{Z}}^2$ by analysis. However, it is contradict to Theorem A of \cite{Wil:2017}, which finishes the proof of Corollary 3.
\end{proof}

\bibliographystyle{amsplain}

\end{document}